\documentclass[a4paper,12pt]{amsart}
\usepackage{amsfonts}
\usepackage{amssymb}
\usepackage{ifthen}
\usepackage{amscd}
\usepackage{amsxtra}
\usepackage{graphicx}
\usepackage{xcolor}

\usepackage{newtxtext,newtxmath}

\usepackage[top=35mm, bottom=35mm, left=30mm, right=30mm]{geometry}

\numberwithin{equation}{section}
\newtheorem{thm}{Theorem}[section]
\newtheorem{lem}[thm]{Lemma}

\theoremstyle{definition}
\newtheorem{defn}[thm]{Definition}

\newtheorem{ques}[thm]{Question}
\newtheorem{rem}[thm]{Remark}

\newcommand{\bbc}{\mathbb{C}}
\newcommand{\bbs}{\mathbb{S}}
\newcommand{\bbr}{\mathbb{R}}
\newcommand{\bbn}{\mathbb{N}}
\newcommand{\bbz}{\mathbb{Z}}

\newcommand{\degg}{\mathrm{deg}}

\newcommand{\fix}{\mathrm{Fix}}
\newcommand{\per}{\mathrm{Per}}

\newcommand{\lift}{\mathcal{L}}

\newcommand{\calc}{\mathcal{C}}
\newcommand{\calu}{\mathcal{U}}
\newcommand{\calv}{\mathcal{V}}

\newcommand{\htop}{h_{\mathrm{top}}}
\newcommand{\css}{C(\bbs,\bbs)}
\newcommand{\lcm}{\mathrm{lcm}}

\usepackage{hyperref}

\begin{document}

\title{Some properties of circle maps with zero topological entropy}
\author{Yini Yang} 
\address{Department of Mathematics,
	Shantou University, Shantou, Guangdong, 515063, China}
\email{ynyangchs@foxmail.com}
\subjclass[2010]{37E10, 37B40, 54H20}

\begin{abstract}
For a circle map $f\colon\bbs\to\bbs$ with zero topological entropy,
we show that
a non-diagonal pair $\langle x,y\rangle\in \bbs\times \bbs$ is non-separable if and only if it is an IN-pair if and only if it is an IT-pair.
We also show that if a circle map is topological null then the maximal pattern entropy of every open cover is of polynomial order.
\end{abstract}

\keywords{circle maps; non-separable pair; IN-pair; IT-pair; 
	topological null; maximal pattern entropy}

\maketitle

\section{Introduction}
The study of the complexity of dynamical system is still an important topic nowadays. 
Different versions of chaos were proposed to represent the complexity in various senses. We refer the reader to the survey \cite{LY16} and references therein for more aspects and details.

For interval maps, there are many known results.
Here we want to name a few. 
In 1975, in their seminal work \cite{LT75}, Li and Yorke showed that
if an interval maps has a periodic point with period $3$ then it is chaotic. 
Since then, the phenomenon introduced in \cite{LT75} is called 
Li-Yorke chaos.
In 1986, Smital \cite{Smital 1986} showed that for an interval map with zero topological entropy, the map is Li-Yorke chaotic if and only if it has a non-separable pair.
In 1989, Kuchta and Smital \cite{M. Kuchta 1989} showed that for a continuous interval map, one scrambled pair
implies Li-Yorke chaos. 
In 1991, Franzova and Smital \cite{N. Franzova 1991} 
showed that a interval map is Li-Yorke chaotic if and only if it has positive topological sequence entropy. 
We refer the reader to the book \cite{R17} for more details on this topic.

Some of the results for interval maps 
have been extended to circle maps or even graph maps. 
In 1990, Kuchta \cite{M. Kuchta 1990} showed that for a circle map, one scrambled pair implies Li-Yorke chaos.
In 2000, Hirc \cite{R. Hric 2000} showed that a circle map is Li-Yorke chaotic if and only if it has positive
topological sequence entropy.
In 2014, Ruette and Snoha \cite{S. Ruette 2014} showed that for a graph map, one scrambled pair implies Li-Yorke chaos.
In 2017, Li, Oprocha, Yang and Zeng \cite{Li J 2017} showed that 
a graph map is Li-Yorke chaotic if and only if it has positive topological sequence entropy. 

In the framework of so-called "local entropy theory", 
lots of notions were introduced to describe specific dynamical properties, see \cite{GY09} for a recent survey.
Among them there are IN-pairs and IT-pairs.
In 2011, Li \cite{Li J 2011} proved that for an interval map with zero topological entropy, a non-diagonal pair is non-separable which is related to Li-Yorke chaos if and only if it is an IN-pair if and only if it is an IT-pair. 
In 2017, Li, Oprocha, Yang and Zeng \cite{Li J 2017} proved that a graph map is Li-Yorke chaotic if and only if it has an IN-pair if and only if it has an IT-pair. 
The authors in \cite{Li J 2017} also proposed an open question as follows. 

\begin{ques}
Let $f\colon G\rightarrow G$ be a graph map with zero topological entropy. Is it true that:
$\left\langle x, y\right\rangle $ is an IN-pair if and only if $\left\langle x, y\right\rangle $ is an IT-pair?
\end{ques}

One of the main results of this paper is to answer this question for circle maps. 
Actually, we have the following theorem.

\begin{thm}\label{thm:main-result1}
	Let $f\colon \bbs\to\bbs$ be a circle map with zero topological entropy and $x\neq y\in \bbs$. Then the following conditions are equivalent:
	\begin{enumerate}
	\item $\left\langle x, y\right\rangle$  is a non-separable pair;
	\item $\left\langle x, y\right\rangle$ is an IT-pair;
	\item $\left\langle x, y\right\rangle$ is an IN-pair.
	\end{enumerate}
\end{thm}

In 2009, Huang and Ye \cite{Huang W Ye X} proposed the concept of maximal pattern entropy. 
They proved that the maximal pattern entropy of a dynamical system is equivalent to the supremum of all topological sequence entropy. 
And they also proved that the maximal pattern entropy can take only discrete values. 
In \cite{Huang W Ye X}, Huang and Ye conjectured that for each topological null system, the  maximal pattern entropy is of polynomial order for every open cover. 
And they proved that the 
conjecture holds for a zero dimensional system. 
In 2011, Li \cite{Li J 2011} proved that the conjecture holds for interval maps. 

Another main result of our paper is to show that the conjecture holds for circle maps by proving the following theorem.

\begin{thm}\label{thm:main-result2}
 A circle map $f\colon \bbs\to\bbs$ is  topological null if and only if $p^*_{\bbs,\calu}$ is of polynomial order for any open cover $\calu$ of $\bbs$.
\end{thm}

The paper is organized as follows. In Section $2$, we recall some basic notions of topological dynamical systems
which will be used latter.  In Section $3$, several concepts and important lemmas related to circle maps are listed, We emphasize properties of the circle maps without periodic points and the structure of $\omega$-limit sets. In Section $4$, we study the relationship of IN-pairs, IT-pairs and non-separable pairs, and then prove Theorem~\ref{thm:main-result1}. In Section $5$, we review some properties on maximal pattern entropy and prove Theorem~\ref{thm:main-result2}.

\section{Preliminaries}

Throughout this paper, let $\bbn$, $\bbr$ and $\bbc$ denote the set of all non-negative integers, real numbers, and  complex numbers, respectively. Let $\bbs$ be the set $\{z\in \bbc\colon |z|=1\}$. Let $M$ denote a non-degenerate closed interval $I$ or the circle $\bbs$ and let $C(M,M)$ denote the set of all continuous maps of $M$ into itself. Denote by $\Delta$ all the diagonal pairs of $X^{2}$ where $X$ is a compact metric space. The length of an interval $I$ is denoted by $|I|$. The cardinality of a finite set $A$ is denoted by $\#(A)$. Denote by $\{a\}$ to take the fractional part of the number $a$. 
%The length of an interval $I$ is denoted by $|I|$. 
%The length of a complex number $z$ is denoted by $\lVert z \rVert$.

Now we introduce some basic notions in topological dynamics. We start with topological dynamical system and then introduce some concepts, including $\omega$-limit set, periodic points, and eventually periodic points. 
%Besides, we introduce the definition of topological semi-conjugate and the relationship between the $\omega$-limit sets of the two systems.  
%At the end of this subsection we introduce topological (sequence) entropy, IN-pair and IT-pair.

\begin{defn}
By a  topological dynamical system, we mean a pair $(X,f)$, where $X$ is a compact metric space with a metric $d$ and $f$ is a continuous map.
\end{defn}

\begin{defn}
	Let $(X,f)$ be a topological dynamical system. A point $x\in X$ is called \emph{periodic} 
	(denoted by $x\in \per(f)$) with \emph{period} $n$
	if $f^n(x)=x$ and $f^i(x)\neq x$ for any $1\leq i<n$. The set of periods of $f$ is denoted by $P(f)$ that is if $x\in \per(f)$ has period $n$ then $n\in P(f)$.
	A periodic point with period $1$ is called a \emph{fixed point} (denoted by $x\in \fix(f)$).	  
\end{defn}

\begin{defn}
	Let $(X,f)$ be a topological dynamical system and $x\in X$.
	We define the $\omega$-limit set of $x$ as
	\[\omega (x,f)=\bigcap_{n\geq 1}\overline{\left\lbrace f^i(x)\colon i\geq n\right\rbrace }\] 
	and the $\omega$-limit set of $f$ as
	\[\omega(f)=\bigcup_{x\in X} \omega(x,f).\]
\end{defn}

\begin{defn}
	Let $(X,f)$ and $(Y,g)$ be two topological dynamical systems. 
	A continuous map $\pi\colon X\rightarrow Y$ is a \emph{semi-conjugacy} between $f$ and $g$
	if $\pi$ is onto and $\pi\circ f=g\circ \pi$. If in addition $\pi$ is a  homeomorphism, then $\pi$ is called a \emph{conjugacy} between $f$ and $g$.
\end{defn}

\begin{defn}
	Let $X$ be a compact metric space and $\calc_X^o$ 
	be the set of open covers of $ X$.
	For $\calu\in \calc_X^o$, we define $N(\calu)$ as the minimum cardinality of subcovers of $\calu$.
	The \emph{join} of two open covers $\calu,\calv\in \calc_X^o$ is
	\[
	\calu\vee\calv=\{U\cap V\colon U\in\calu, V\in\calv\}.
	\]
\end{defn}

\begin{defn}
	Let $(X,f)$ be a topological dynamical system and $\calu,\calv\in \calc_X^o$. \emph{The topological entropy of $f$ with respect to $\calu$} is defined as  
	\[\htop(f,\calu)=\lim_{n \to \infty}\frac{1}{n}\log\biggl (N\bigg (\bigvee_{i=1}^{n}f^{-i}(\calu)\biggr )\bigg ),\]  
	and \emph{the topological entropy of $(X,f)$} is defined as 
	\[\htop(f)=\sup_{\calu\in\calc_X^o}\htop(f,\calu)\]   
	\end {defn}
	
	\begin{defn}
		Let $(X,f)$ be a topological dynamical system and $A=\left\lbrace a_{1}<a_{2}<a_{3}<\cdots\right\rbrace $ an increasing sequence of positive integers. Take $\calu\in \calc_X^o$. \emph{The topological sequence entropy of $f$ with respect to $\calu$ along the sequence $A$} is defined as 
		\[\htop ^A(f,\calu)=\lim_{n \to \infty}\frac{1}{n}\log \bigg (N \bigg (\bigvee_{i=1}^n f^{-a_i}(\calu) \bigg ) \bigg ),\]
	     \emph{the topological sequence entropy of $(X,f)$ along the sequence $A$} is defined as
		\[\htop^A(f)=\sup_{\calu\in\calc_X^o}\htop^A(f,\calu).\] 
		and \emph{the topological sequence entropy} of $(X,f)$ is 
		\[
		\htop^\infty(f)=\sup_A \htop^A(f),
		\]
		where the supreme takes over all increasing sequences of positive integers. 
		
		Moreover,  $(X,f)$ is called \emph{topological null} if $\htop^\infty(f)=0$ and 
		otherwise it is called \emph{non-null}. 
	\end{defn}

The following two lemmas can be verified easily by the definition of topological semi-conjugacy.

\begin{lem}\label{lem:W-limit set-conjugate} %\cite{M. Kuchta 1990}
	Let $(X,f)$ and $(Y,g)$ be two topological dynamical systems and  $\pi\colon X\rightarrow Y$ is a \emph{semi-conjugacy} between $f$ and $g$.
	It holds that for any $x\in X$,  $\pi (\omega(x,f))=\omega(\pi(x),g)$.
\end{lem}

\begin{lem}\label{lem: conjugate and entropy}
	Let $\varphi\colon (X,f)\to (Y,g)$ be a topological semi-conjugacy. Then $\htop(g)\leq \htop(f)$ and $\htop^{\infty}(g)\leq \htop^{\infty}(f)$.
\end{lem} 
%The following result characterizes the relationship between Li-Yorke chaos and sequence entropy of both interval maps and circle maps.  
%\begin{lem}[\cite{Canovas 2001}]  \label{lem:h-infty-chaotic-circle}
%	Let $f\in C(M,M)$. Then
%	\begin{enumerate}
%		\item 
%		$\htop(f)>0$ if and only if $\htop^{\infty}(f)=\infty$;
%		\item 
%		$\htop(f)=0$ and $f$ is Li-Yorke chaotic if and only if $\htop^{\infty}(f)=log 2$; 
%		\item 
%		$f$ is not Li-Yorke chaotic if and only if  $\htop^{\infty}(f)=0$.
%	\end{enumerate}
%\end{lem}

\begin{defn}
	Let $(X,f)$ be a topological dynamical system and 
	$ A_1, A_2,\cdots, A_k \subset X$.
	We call $I\subset\bbn$ 
	\emph{an independence set} of $\left\lbrace A_1, A_2,\cdots, A_k \right\rbrace$ if for any non-empty finite subset $J$ of $I$ and $S\in \{1,2,\dotsc,k\}^J$, 
	$\bigcap\nolimits_{i\in J}f^{-i}A_{S(i)}\neq \emptyset$.
\end{defn}

\begin{defn}
	Let $(X,f)$ be a topological dynamical system.
	A pair $\left\langle x, y\right\rangle  \in X\times X$ is called an \emph{IN-pair} (reps.\ an \emph{IT-pair}) if for any neighborhoods $U_1$ and $U_2$
	of $x$ and $y$ respectively, $\left\lbrace U_1, U_2 \right\rbrace $ has arbitrarily large finite independence sets 
	(reps.\ $\left\lbrace U_1, U_2 \right\rbrace $ has an infinite  independence set).
	Denote the set of IN-pairs and IT-pairs of $(X,f)$ by
	$IN(X,f)$ and $IT(X,f)$ respectively.
	
	It is clear that every IT-pair is also an IN-pair.
\end{defn}

\begin{thm}[\cite{Kerr 2007}]\label{thm:null-IN}
A topological dynamical system is topological null if and only 
if every IN-pair is diagonal.	
\end{thm}

\begin{lem}[\cite{Kerr 2007}]\label{INIT-conjugate}
	Let $(X,f)$ and $(Y,g)$ be two topological dynamical systems and  $\pi\colon X\rightarrow Y$ is a semi-conjugacy between $f$ and $g$.       
	Then $\pi\times \pi (IN(X,f))=IN(Y,g)$ 
	and $\pi\times \pi (IT(X,f))=IT(Y,g)$.
\end{lem}

The following  is a ``folklore'' result, see e.g. \cite[Theorem 2.6]{TYZ10} but without proofs.
Here we provide a proof for completeness.

\begin{lem}\label{IN-IT-p}
Let $(X,f)$ be a topological dynamical system. 
Then for every $p\in\bbn$, $IN(X,f)=IN(X,f^p)$ and  $IT(X,f)=IT(X,f^p)$.
\end{lem}

\begin{proof}
	We only prove the case $IN(X,f)=IN(X,f^p)$, as the case $IT(X,f)=IT(X,f^p)$ is similar.

$(\Rightarrow)$ 
Let $\left\langle x,y \right\rangle $ be an IN-pair of $(X,f)$. 
By the definition, for any neighborhoods $U_1$ of $x$ and $U_2$ of $y$ respectively and any $m\in\bbn$, $\{U_1, U_2\}$ has an independence set $I$ with length $mp$, since the number of the remainders of mod $p$ is finite, if we denote $Q_r$ as the set of all numbers of the independece set $I$ which mod $p$ has the remainder $r$, there exist at least a $r\in \{0,1,\cdots, p-1\}$ such that the number of $Q_r$ is larger than or equal to $m$. We choose one $r$ which satisfies this condition and choose $m$ elements  $\{q_1,q_2,\cdots, q_m\}$ from $Q_r$. Denote $I_{m}=\{l_1,l_2,\cdots, l_m\}$ such that for all 
$i\in\left\lbrace 1,2,\cdots, m \right\rbrace $, $q_i=l_i+r$.  
Now we will show that 	$I_m$ is an independence set of $\{U_1, U_2\}$ under $f^p$ which has length $m$.
For all $J_{m}\subset I_{m}$ which $J_{m}$ is a finite set, 
there exists a finite set $J\subset I$ such that $J=p\cdot J_m +r$. 
For any $\sigma=\left\lbrace 1,2 \right\rbrace^{J_m}$, 
\[
f^{-r}\biggl(\bigcap_{i\in J_m}(f^p)^{-i}U_{\sigma(i)}\biggr)=\bigcap_{i\in J_m}f^{-r}(f^{-pi}U_{\sigma(i)})=\bigcap_{i\in J_m}f^{-r-pi}U_{\sigma(i)}=\bigcap_{j\in J}f^{-j}U_{\sigma(i)}.
\]
Take $\sigma'=\left\lbrace 1,2 \right\rbrace^{J}$ such that
for all $j=r+pi\in J$, $\sigma'(j)=\sigma(i)$. 
Since $I$ is an independence set of $f$, 
\[\bigcap_{j\in J}f^{-j}U_{\sigma^{'}(j)}=\bigcap_{j\in J}f^{-j}U_{\sigma(i)}\neq \emptyset. \]
Therefore $f^{-r}(\bigcap_{i\in J_m}(f^p)^{-i}U_{\sigma(i)})\neq \emptyset$ and $\bigcap_{i\in J_m}(f^p)^{-i}U_{\sigma(i)}\neq \emptyset$ obviously.
Thus $\left\langle x,y \right\rangle $ is an IN-pair of $(X,f^p)$.
	
$(\Leftarrow)$ 
Let $\left\langle x,y \right\rangle $ be an IN-pair of $(X,f^p)$.
By the definition, for any neighborhoods $U_1$ of $x$ and $U_2$ of $y$ respectively and any $m\in\bbn$, there exists an independence set $I_m$ of $\{U_1, U_2\}$ under $f^p$ with length $m$.
 Now we will claim that $I_m'=pI_m$ is an independence set of 
$\{ U_1, U_2 \} $ under $f$ with length $m$.
	For all $J_m'\subset I_m'$, 
	there exists $J_m \subset I_m$ such that $J_m'=pJ_m$. 
	For any $\sigma'=\{1,2\}^{J_m'}$, there exists $\sigma=\{1,2\}^{J_m}$ such that for all
	 $j=pi\in J_m'$, 
	$\sigma'(j)=\sigma(i)$. Then 
	\[\bigcap_{j\in J_m'}f^{-j}U_{\sigma'(j)}
	= \bigcap_{i\in J_m}f^{-pi}U_{\sigma(i)}
	= \bigcap_{i\in J_m}(f^p)^{-i}U_{\sigma(i)}\neq \emptyset.\]
Hence $\left\langle x,y \right\rangle $ is an IN-pair of $f$.
\end{proof}

\section{Preparations on circle maps}
\subsection{General notions}
In this subsection we introduce natural projection, the ordering and metric of the circle, and 
orientation-preserving homeomorphism. Besides we introduce the definition of lifting, degree and rotation number of the circle.
We refer the readers to the textbook \cite{Alseda 2001} and \cite{Boris Hasselblatt 2003} for more informations. 

\begin{defn}
The \emph{natural projection}
	 $e\colon \mathbb{R}\rightarrow \bbs$ is the map $e(x)=\exp(2\pi i x)$.
	 
According to the natural projection, we also view the circle $\bbs$
as $[0,1)\pmod 1$.
\end{defn}

\begin{defn}\label{defn:metric of circle} 
	%\cite{Boris Hasselblatt 2003}
	We define a metric $d$ on $\bbs$ as follows: for any $x,y\in \bbs$, 
	\[d(x,y)=\min\left\lbrace \left| x'-y'\right| \colon e(x')=x,  e(y')=y\right\rbrace. \]
	It is easy to see that the distance between $a$ and $b$ can be described as follows 
	\[d(a,b)=\begin{cases}
	\left| a-b \right|, \quad\left| a-b \right|\leq\frac{1}{2},\\
	1-\left| a-b \right|, \quad  \left| a-b \right|>\frac{1}{2}.
	\end{cases}\]
\end{defn}

\begin{defn}
	\label{defn:order of circle} 
	Given $x_0, x_1, x_2, \cdots, x_n \in \bbs$, take
	$\tilde{x}_0, \tilde{x}_1, \tilde{x}_2, \cdots, \tilde{x}_n \in [\tilde{x}_0,\tilde{x}_0+1)\subset\mathbb{R}$ such that $e(\tilde{x}_i)=x_i$, 
	\emph{the ordering of $(x_0, x_1, x_2, \cdots, x_n)$ on $\bbs$} is the 
	permutation of $\left\lbrace 1,2, \cdots ,n\right\rbrace $ which satisfies that
	$\tilde{x}_0<\tilde{x}_{\sigma_{(1)}}<\tilde{x}_{\sigma_{(2)}}, \cdots< \tilde{x}_{\sigma_{(n)}}<\tilde{x}_0+1$.
\end{defn}

\begin{defn}
	A circle map $f\in \css$ is called \emph{orientation-preserving homeomorphism} if $f$ is a homeomorphism and it preserves the ordering which is defined in Definition~\ref{defn:order of circle} under $f$.
\end{defn}

The following theorem is well known, see the textbook~\cite{Boris Hasselblatt 2003} for more detail.

\begin{thm} %\cite{Boris Hasselblatt 2003}
	Let $f\in \css$. We call a continuous function $F\colon \mathbb{R}\rightarrow\mathbb{R}$ a \emph{lifting} of $f$ if  
	$e\circ F=f\circ e$. 
	The integer $d\in\bbz$
	such that $F(x+1)=F(x)+d$ exists for all $x\in \bbr$ and is called the \emph{degree} of $f$ (or $F$), denoted by $\degg(f)$.
	
	In later parts of this paper, we denote by $\lift_1$ the set of all liftings of the circle maps with $\degg(f)=1$.
\end{thm}

\begin{defn} %[\cite{Alseda 2001}, Chapter III, Section 7]
	Let $F\in \lift_1$ and $x\in \mathbb{R}$. We define $\overline{\rho}(x)$ and $\underline{\rho}(x)$ as follows $\colon$
	\[\overline{\rho}(x)=\limsup_{n\to\infty}\frac{F^n(x)-x}{n}\quad\text{and}\quad \underline{\rho}(x)=\liminf_{n\to\infty}\frac{F^n(x)-x}{n}\]
	When $\overline{\rho}(x)=\underline{\rho}(x)$ we write only $\rho(x)$. The number
	$\rho(x)$ is called the \emph{rotation number of $x$ with respect to $F$}.
\end{defn}

The following lemma is classical, 
see  e.g.\ \cite[Chapter III, Section 7]{Alseda 2001} for details.
\begin{lem}
	Let $F\in \lift_1$ and be a non-decreasing map. Then for any $x\in\mathbb{R}$, the limit
	\[\lim_{n\to\infty}\frac{F^n(x)-x}{n}\]
	exists and is independent of $x$. Therefore in the later parts of our paper, we can define $\rho(F)$ as $\rho_F(x)$ for any $x\in \bbr$. Moreover, the limit is a rational number if and only if $f$ has periodic points.
\end{lem}

\begin{rem}
Let $f$ be an orientation preserving homeomorphism. It is easy to see that for any two liftings $F_1$ and $F_2$ of $f$, $\rho(F_1)-\rho(F_2)\in \bbz$. Therefore we can define $\rho(f)=\{\rho(F)\}$ for any lifting $F$.
\end{rem}
\begin{lem}[{\cite[Lemma 4.7.1]{Alseda 2001}}]\label{lift-entropy} 
Let $f\in \css$ and $F$ be a lifting of $f$.
Assunme that there exists a closed interval $I$ invariant for $F$. Then $\htop(F|_I)\leq\htop(f)$.
\end{lem}

\begin{rem}
If $|I|\geq 1$, by lemma~\ref{lem: conjugate and entropy}, we get the conclusion that 
$\htop(F|_I)=\htop(f)$
\end{rem}

In order to describe the properties of the periodic points of interval maps and circle maps, we introduce the following definition.

\begin{defn}
	Let $\bbn_{sh}=\bbn\cup \{2^\infty\}$.
	We define an order of the numbers in $\bbn_{sh}$ as follow: 
	\begin{align*}
	&3 \succ 5\succ 7\succ \cdots  \\
	&\succ2 \cdot 3\succ 2 \cdot 5\succ2 \cdot 7\succ \cdots\\
	 &\succ 2^2 \cdot 3 \succ2^2 \cdot 5 \succ2^2 \cdot 7\succ  \cdots\\
	 &\cdots\\
	 &\succ2^\infty \succ  \cdots \succ2^n \succ  \cdots \succ4\succ2 \succ1,
	\end{align*}
 and call $(\bbn_{sh},\succ)$ \textit{Sharkovsky order}.
	
	Let $a\in \bbn_{sh}$ and denote $ S(a)=\{ m\in\bbn_{sh}\colon a\succeq m\}$.
\end{defn}

%\textcolor{red}{Check the exact page or theorem number of the following three results in \cite{Alseda 2001}!}
The following result is the well-known Sharkovsky theorem.

\begin{thm}[see e.g.\@ \cite{Alseda 2001}]\label{sharkovskii theorem}
	Let $f\in C(I,I)$, then there exists $n\in \bbn_{sh}$ such that $P(f)=S(n)$.
\end{thm}

For interval maps with zero topological entropy, we characterize the properties of periodic points as follow.

\begin{thm}[see e.g.\@ \cite{Alseda 2001}] \label{thm:interval-entropy-0}
	Let $f\in C(I,I)$ with $\htop(f)=0$. 
	Then there exists $n\in \bbn_{sh}$ such that $2^\infty \succeq n$ and $P(f)=S(n)$.
\end{thm}

The properties of periodic points of circle maps are more complicated, we just consider the case of circle maps with zero topological entropy.

In the following theorem, see Theorem 3.5.3 and Theorem 3.6.8 of the textbook\cite{Alseda 2001} when $\degg(f)=0$ and $\degg(f)=-1$ $\degg(f)=-1$. About the case when $\degg(f)=1$, we refer the readers to Corollary 3.9.7 and Theorem 3.10.1.  
\begin{thm} \label{thm:deg-1-period}
	Let $f\in \css$ with $|\deg(f)|\leq 1$, $\htop(f)=0$ and $\per(f)\neq\emptyset$.
	Then there exist $k\in\bbn$ and $n\in \bbn_{sh}$ such that $2^\infty \succeq n$
	and 
	$P(f)=kS(n)$.
	Moreover if $\deg(f)=0$ or $\deg(f)=-1$, then $f$ has a fixed point, therefore $k=1$.
\end{thm}

\begin{rem}\label{rem:deg-1-period-interval}
	Let $f\in \css$ with $|\deg(f)|\leq 1$ and $\htop(f)=0$. By Lemma~\ref{thm:deg-1-period}, there exist $k\in\bbn$ and $n\in \bbn_{sh}$ 
	such that $2^\infty \succeq n$
	and $P(f)=kS(n)$. In other words the minimal period of the periodic points of $f$ is $k$, and the other periods are all multiple of $k$.
	If $J\subset \bbs$ is a periodic interval with period $m>1$,
	then there exists $x\in J$ such that $x$ is a perodic point with period $m$. Thus $m$ is also the multiple of $k$.
\end{rem}

\subsection{Circle maps without periodic points}\label{subsection:Circle maps}
In this subsection we recall some properties of the circle maps without periodic points. 
%Theorem~\ref{thm:Poincare} characterizes the relationship between the circle maps with irrational rotation number and its correlative rotation. 
%Theorem~\ref{lem:not-homermorphism and nowhere dense limit set}  characterizes the properties of the circle maps which is not a homeomorphism and has no periodic points. 

\begin{defn}
Let $\alpha\in [0,1)$ and $R_{\alpha}\colon [0,1)\rightarrow [0,1)$,  $x\rightarrow x+\alpha \pmod 1$. We call $R_{\alpha}$ the rotation of
the circle with rotation number $\alpha$.  
\end{defn}

\begin{lem}[\cite{J. Auslander}, see also {\cite[Section 4.1]{Boris Hasselblatt 2003}}]\label{thm:Poincare} 
Let $f\in\css$ without periodic points. Assume that $f$ is a homeomorphism. Then for any $ x,y\in \bbs$, $\omega(x,f)=\omega(y,f)$. Denote $W=\omega(x,f)$ and $\rho(f)$ the rotation number of $f$, 
$\rho(f)\notin \mathbb{Q}$ since $f$ has no periodic points. In this case $W$ is a perfect set and exactly one of the following alternatives holds $\colon$
	\begin{enumerate}
		\item if $f$ is topological transitive, $W=\bbs$ and $f$ topological conjugate to $R_{\rho(f)}$.
		\item if $f$ is not topological transitive, $W$ is a nowhere dense set and $f$ topological semi-conjugate to $R_{\rho(f)}$.
	\end{enumerate} 
\end{lem}

%\textcolor{red}{I think the following result should not cite  \cite{S. Kolyada}. It must be proved in somewhere early.}
%I think we should cite \cite{J. Auslander} Theorem1 and Theorem2

\begin{lem}[\cite{J. Auslander}]\label{lem:not-homermorphism and nowhere dense limit set}
	Let $f\in\css$ without periodic points. Assume that $f$ is not a homeomorphism. Then for any $x,y\in \bbs$, $\omega(x,f)=\omega(y,f)$. Denote $W=\omega(x,f)$. Then $W$ is a nowhere dense perfect set. Moreover, $f|_W$ is orientation preserving or orientation reserving.
\end{lem}

\begin{lem}
[{\cite[Theorem 1]{R. Hric 2000}}]\label{lem:f-entropy-no-period-point}
Let $f\in\css$. Then $f$ is not Li-Yorke chaotic if and only if $f$ is topological null. In particular if $f$ has no periodic points, then $f$ is topological null.
\end{lem}

\subsection{The structure of $\omega$-limit sets of circle maps with zero topological entropy}
%In this section we introduce the structure of $\omega$-limit sets of both circle maps and interval maps. 
%We give the definition of periodic intervals and Lemma~\ref{lem:omega-limit-periodic-interval} which characterizes the relationship between periodic intervals and $\omega$-limit set. 
%We also give the definition of periodic portion and Lemma~\ref{lem:omega-limit-periodic-interval} which characterizes the relationship between periodic portion and $\omega$-limit set.

\begin{defn}
	Let $f\in C(M,M)$.
	A closed subinterval $J$ of  $M$ is called \emph{periodic}
	if there exists a positive integer $n$ such that 
	$J$, $f(J)$, $\dotsc$, $f^{n-1}(J)$ are pairwise
	disjoint and $f^n(J)=J$.
	In this case, $n$ is the \emph{period} of $J$ and $\lbrace f^i(J)\rbrace_{i=1}^{n-1} $ 
	is called a \emph{cycle of intervals}.
\end{defn}

%\textcolor{red}{The following two lemmas without reference, please add.}

%I'm afraid \cite{Block 1992} is not the original paper. 
%I can't find the original one, maybe it is one of the papers of Sharkovsky.
The following result is due to Sharkovsky,
see e.g.\@ Proposition 5.23 and Remarks on graph maps in page 132 of \cite{R17}.
\begin{lem}\label{lem:omega-x}
	Let $f\in C(M,M)$ with $\htop(f)=0$.
	For any $x\in M$, either $\omega(x,f)$ is a periodic orbit
	or $\omega(x,f)$ is infinite and $\omega(x,f)\cap \per(f)=\emptyset$.
\end{lem}

The following result can be easily confirmed by the definition.

\begin{lem}\label{lem:omega-limit-periodic-interval}
	Let $f\in C(M,M)$ with $\htop(f)=0$ and $J$ be a periodic subinterval of $M$ with period $n$.
	For any $x\in M$, either $\omega(x,f)\subset \bigcup_{i=0}^{n-1}f^i(J)$ or $\omega(x,f)\cap \bigcup_{i=0}^{n-1}f^i(J)=\emptyset$. 
\end{lem}

%\textcolor{red}{The following  lemma should cite Blokh's paper on the structure of $\omega$-limit sets of graph maps.}

We will need the following structure of $\omega$-limit sets of interval maps and circle maps.

\begin{lem}[\cite{Blokh90a}] \label{rem: period portion}
	Let $f\in C(M,M)$ with $\htop(f)=0$ and $\per(f)\neq\emptyset$. For any $x\in M$, if $\omega(x,f)$ is infinite, then there is a period portion of the $\omega$-limit set $\omega(x,f)$, i.e. 
	there exists a sequence of periodic intervals $ (J_n)_{n\in\bbn}$ with periods $(k_n)_{n\in\bbn}$ such that$\colon$
	\begin{enumerate}
		\item $k_1\geq 1$,  and $k_{n+1}$ is the multiple of $k_n$ for all $n\geq 1$; 
		\item $J_{n+1}\subset J_{n}$, and every connected components of $\bigcup_{i=0}^{k_n-1} J_n$ contains the same numbers of the connected components of $\bigcup_{i=0}^{k_{n+1}-1} J_{n+1}$;
		\item  $\omega(x,f)\subset \bigcap_{n\geq 1} \bigcup_{i=0}^{k_n-1} f^i(J_n)$ and  $\omega(x,f)$ does not contain periodic points.
	\end{enumerate}
\end{lem}

\begin{lem}\label{lem:periodic-splitting-intersect}
Let $f\in C(M,M)$ with $\htop(f)=0$ and $\per(f)\neq\emptyset$.
Assume that $x\in M$ and $\omega(x,f)$ is infinite.
Let $(J_n)_{n\in\bbn}$ and $(k_n)_{n\in\bbn}$ 
provided by Lemma~\ref{rem: period portion}.
Assume that $K$ is a periodic  subinterval of $M$ with period  $m$
and $\omega(x,f)\cap \bigcup_{i=0}^{m-1}f^i(K)\neq\emptyset$.
Then there exists some $n\in\bbn$ such that 
$\bigcup_{i=0}^{k_n-1}f^i(J_n)\subset \bigcup_{i=0}^{m-1}f^i(K)$.
\end{lem}
\begin{proof}
As $\{k_n\}$ is strictly increasing, $\bigcap_{n=1}^\infty \bigcup_{i=0}^{k_n-1}f^i(J_n)$ has uncountably many connected components 
and at most countably many
of them can be non-degenerate.
Moreover, each connected component  intersects $\omega(x,f)$.
By Lemma~\ref{lem:omega-limit-periodic-interval}, 
$\omega(x,f)\subset \bigcup_{i=0}^{m-1}f^i(K)$.
So there exists $i_0\in\{0,1,\dotsc,m-1\}$ such that $f^{i_0}(K)$ contains at least three
non-degenerate  connected components of $\bigcap_{n=1}^\infty \bigcup_{i=0}^{k_n-1}f^i(J_n)$.
Then there exist $n\in\bbn$ and $j_0\in\{0,1,\dotsc,k_n-1\}$ 
such that $f^{j_0}(J_n)\subset  f^{i_0}(K)$.
As $J_n$ and $K$ are periodic, we have
$\bigcup_{i=0}^{k_n-1}f^i(J_n)\subset \bigcup_{i=0}^{m-1}f^i(K)$. 
\end{proof}

\subsection{Some auxiliary lemmas of circle maps}

Our purpose is to study the properties of circle maps with zero topological entropy. At first of this subsection, we use the following lemmas to illustrate that
we can restrict our study object to the circle maps with $|\degg(f)|\leq 1$, more 
precisely, non-extensible circle maps with $|\degg(f)|\leq 1$.

First we define extensible and non-extensible circle maps. This concept was first proposed by Mai in \cite{Mai J 1997}. Here we modify the original definition in \cite{Mai J 1997} according to what we need and describe as follow.
\begin{defn} 
	Let $f\in \css $ and $F$ be a lifting of $f$.
	We call	$f$ \emph{extensible} if there exists $r\in\bbr$ and $n\in\bbn$ such that  $|F^n([r,r+1])|\ge \left| \degg(f)\right| +1$. 
	Otherwise it is called \emph{non-extensible}.
\end{defn}

In order to investigate the relationship between circle maps with zero entropy and the
non-extensible circle maps with $|\degg(f)|\leq 1$, we introduce horseshoe and the relationship between horseshoe and positive entropy. 
  
\begin{defn}\label{defn:cicle-horseshoe}
	Let $f\in C(M,M)$. We say that two closed non-degenerate subintervals $K_{1}$ and $K_{2}$ of $M$ with disjoint interiors form a \emph{horseshoe} if there exist some $n\in \mathbb{N}$ and subintervals $K_{i1}$ and $K_{i2}$ of $K_{i}$ such that $f^n(K_{i1})=K_{1}$ and $f^n(K_{i2})=K_{2}$ for $i=1,2$. When $M$ denote a non-degenerate closed interval $n=1$.
\end{defn}

\begin{lem}[\cite{LM93}]
	\label{lem:horseshoe-positive-entorpy}
Let $f\in C(M,M)$. If $f$ has a horseshoe then $\htop(f)>0$. 
\end{lem}

The following lemma tells us that a circle map with zero topological entropy is a non-extensible circle map  
with $|\degg(f)|<1$. It is an analogous result with Theorem 1.1 in \cite{Mai J 1997} but still have some differences since the definitions of non-extensible are not the same.   Here we give our proof.

\begin{lem}\label{lem:zero-entropy-non-ext}
Let $f\in \css$ with $\htop(f)=0$. Then $f$ is a non-extensible circle map with $|\degg(f)|\leq 1$.
\end{lem}
\begin{proof}
Let $F$ be a lifting of $f$. 
If $|\degg(f)|\geq 2$, then  $|F([0,1])|\geq |\degg(f)|\geq 2$.
Pick two closed subintervals $K_1$ and $K_2$ of $[0,1]$ 
with disjoint interiors such that $|F(K_1)|=|F(K_2)|=1$, which implies that $eF(K_1)=eF(K_2)=\bbs$ and as a result $e(K_1)$ and $e(K_2)$ form a horseshoe of $f$.
By Lemma~\ref{lem:horseshoe-positive-entorpy}, $\htop(f)>0$ which is a contradiction. Thus $|\degg(f)|\leq 1$.

Consider the case $|\degg(f)|=1$
and assume that $f$ is extensible.
Then there exist $r\in\bbr$ and $n\in\bbn$ such that  $|F^n([r,r+1])|\geq 2$. Analogous to the proof in the situation $|\degg(f)|\geq 2$,  we can induce that $f$ has a horseshoe which is a contradiction. Thus $f$ is non-extensible. 

Now we consider the case $\degg(f)=0$ and assume that $f$ is extensible.
Note that in this case $F$ is a periodic function with period $1$. Pick $x,y\in\bbr$ with $x<y<x+1$ such that $F(x)=\min F(\bbr)$
and $F(y)=\max F(\bbr)$.
As $f$ is extensible, $F(y)-F(x)\geq 1$.
Let $K_1=[x,y]$ and $K_2=[y,x+1]$.
It is easy to see that $e(K_1)$ and $e(K_2)$ form a horseshoe of $f$ which is a contradiction. Thus $f$ is non-extensible.
\end{proof}

For a circle map with zero topological entropy, we will be searching for a lifting $F$ and an interval $I$ such that $F|_I$ become an interval map. The idea was first proposed in \cite{Mai J 1997},
here we restate the result and give our proof according to what we need.

\begin{lem}[\cite{Mai J 1997}]\label{lem:I-extension}
Let $f\in\css$ with $\htop(f)=0$ and $\fix(f)\neq\emptyset$.
Then there exist a lifting $F$ of $f$ and
a closed subinterval $I=[a,b]$ of $\bbr$ with $1\leq b-a<2$ such that $F(I)\subset I$.
Moreover,
\begin{enumerate}
	\item if $\degg(f)=0$, then $b-a=1$ and $|F([a,b])|<1$;
	\item if $\degg(f)=1$, then $F([a,b])=[a,b]$,
	$F([a,b-1])\subset [a,b-1]$ and $F([a+1,b])\subset[a+1,b]$;
	\item if $\degg(f)=-1$, then 
	$F([a,b])=[a,b]$,
	$F([a,b-1])\subset [a+1,b]$ and $F([a+1,b])\subset[a,b-1]$.
\end{enumerate}
\end{lem}
\begin{proof}
By Lemma~\ref{lem:zero-entropy-non-ext}, 
$|\degg(f)|\leq 1
$ and $f$ is non-extensible.
We divide the proof into three cases according to the degree of $f$. 

\textbf{Case 1}: $\degg(f)=0$. 
Let $F$ be a lifting of $f$. 
As $F$ is a periodic function with period $1$, $F$ is bounded.
As $f$ is non-extensible, $\max F(\bbr)-\min F(\bbr)<1$.
Let $a=\min F(\bbr)$. Then $F([a,a+1])\subset [\min F(\bbr),
\max F(\bbr)]\subset [a,a+1]$.

\textbf{Case 2}: $\degg(f)=1$.
	Since $f$ has a fixed point, we can choose a lifting $F$ of $f$ such that $F$ has a fixed point $x$.
	Then $x+1$ is also a fixed point and $F([x,x+1])\supset [x,x+1]$. 
Put $J=[x,x+1]$. By induction, $F^{n+1}(J)\supset F^n(J) \supset J$ for all $n\in\bbn$.  

Let $K=	\overline{\bigcup_{n=1}^{\infty}F^{n}(J)}$.
It is clear that
$K\supset J$ and $F(K)=K$.
Since $f$ is non-extensible, for any $n\in\bbn$, $|F^n(J)|<2$, hence $1\leq|K|\leq 2$.
Denote $K=[a,b]$.
Since $F[a,b]=[a,b]$, 
$\min F|_{[a,b-1]}\geq \min F|_{[a,b]}\geq a$.
If there exists some $y\in [a,b-1]$ such that $F(y)>b-1$, then $y+1\in [a+1,b]\subset [a,b]$ and $F(y+1)=F(y)+1>b$, which contradicts to $F([a,b])=[a,b]$.
Therefore for any $y\in [a,b-1]$, $F(y)\leq b-1$. Thus we have  $F([a,b-1])\subset [a,b-1]$.
Similarly, we have  $F([a+1,b])\subset [a+1,b]$.

If $|K|<2$, we take $I=K$.
If $|K|=2$, that is $b-a=2$, then $b-1=a+1$. 
Since $F([a,b-1])\subset [a,b-1]$, $F([a+1,b])\subset [a+1,b]$, we have $F(b-1)=b-1$ and $F(a+1)=a+1$. Thus $F(a)=a$ and so $F([a,b-1])= [a,b-1]$. In this case, it is enough to take $I=[a,a+1]$. 
 
\textbf{Case 3}:  $\degg(f)=-1$.
Since $f$ has a fixed point, we can choose a lifting $F_1$ of $f$ such that $F_1$ has a fixed point $x$. 
Then $F_1(x+1)=F_1(x)-1=x-1$ and $F_{1}([x,x+1])\supset[x-1,x]$. 
Now we choose another lifting $F=F_{1}+1$ of $f$
and put $J=[x,x+1]$. Then we have $F(J)\supset J$.
By induction, $F^{n+1}(J)\supset F^n(J)\supset J$ for all $n\in\bbn$.

Let $K=	\overline{\bigcup_{n=1}^{\infty}F^{n}(J)}$. 
It is clear that we have 
$K\supset J$ and $F(K)=K$.
Since $f$ is non-extensible, $|F^n(J)|<2$ for all $n\in\bbn$ and $|K|\leq 2$. Denote $K=[a,b]$.

Since $F[a,b]=[a,b]$, $\max F|_{[a,b-1]}\leq b$. If there exists some $y\in [a,b-1]$
such that $F(y)<a+1$, then there exists $y+1\in [a+1,b]\subset [a,b]$ such that $F(y+1)<a$, which is a contradiction to $F([a,b])=[a,b]$. Therefore $F([a,b-1])\subset [a+1,b]$. 
Similarly, we have $F([a+1,b])\subset [a,b-1]$. 

If $|K|<2$, we take $I=K$.
If $|K|=2$, that is $b-a=2$, then $b-1=a+1$. 
Since $F([a,a+1])\subset [a+1,b]$ and 
$F([a+1,b])\subset [a,a+1]$, we have $F(a+1)=a+1$ and $F(a)=F(a+1)+1=b$. It follows that $F([a,a+1])= [a+1,b]$. Now we choose a new
lifting $F-1$ to replace $F$, then we have $(F-1)[a,b-1]= [a,b-1]$ which $|[a,b-1]|=1$. In this case, we take $I=[a,a+1]$. 
\end{proof}

\begin{lem}[\cite{M. Kuchta 1990}]\label{lem:F-Per-f}
	Let $f\in \css$ with $\fix(f)\neq\emptyset$. If $F$ and $I$ are the lifting and interval provided by Lemma~\ref{lem:I-extension}, then $e(\per(F))=\per(f)$.
\end{lem}

The following lemma describes the $\omega$-limit set of $F$, which is the lifting of $f$ with $\htop(f)=0$ and $\degg(f)=1$.

\begin{lem}\label{lem:degree1-w-limit set}
	Let $f\in\css$ with $\htop(f)=0$, $\degg(f)=1$ and $\fix(f)\neq\emptyset$.
	Let $F$ and $I=[a,b]$ provided by Lemma~\ref{lem:I-extension}.
	Then for any $x\in [a,b]$,  $\omega(x,F)\subset [a,b-1]$ or $\omega(x,F)\subset (b-1,a+1)$ or $\omega(x,F)\subset [a+1,b]$.
\end{lem}
\begin{proof}
Since $F([a,b-1])\subset [a,b-1]$ and $F([a+1,b])\subset[a+1,b]$,
if there exists some $n\in\bbn$ such that 
$F^n(x)\in[a,b-1]$ or $F^n(x)\in[a+1,b]$, then  $\omega(x,F)\subset [a,b-1]$ 
or $\omega(x,F)\subset [a+1,b]$. 
Now we assume that for any $n\in\bbn$, $F^n(x)\in(b-1,a+1)$. 
Then $\omega(x,F)\subset [b-1,a+1]$.
If $b-1\in \omega(x,F)$,
then $F(b-1)\in \omega(x,F)\subset [b-1,a+1]$.
Since $F(b-1)\in F([a,b-1])\subset[a,b-1]$, $F(b-1)=b-1$.
By Lemma~\ref{lem:omega-x}, $\omega(x,F)=\{b-1\}$ and then $\omega(x,F)\subset[a,b-1]$.
Similarly, if $a+1\in \omega(x,F)$ then $\omega(x,F)=\{a+1\}\subset[a+1,b]$.
If $b-1,a+1\not\in \omega(x,F)$, then $\omega(x,F)\subset (b-1,a+1)$.
\end{proof}

\begin{lem}\label{lem:F-periodic-subinterval}
	Let $f\in\css$ with $\htop(f)=0$, $\degg(f)=1$ and $\fix(f)\neq\emptyset$.
	Assume that $F$ and $I=[a,b]$ are the lifting and interval provided by~\ref{lem:I-extension}. 
	If $J$ is a periodic subinterval of $I$ with period $n>1$, then 
	exactly one of the following alternatives holds:	
	\begin{enumerate}
		\item $J\subset [a,b-1]$ , $\bigcup_{j=0}^{n-1}F^j(J)\subset [a,b-1]$; 
		\item $J\subset [a+1,b]$ , $\bigcup_{j=0}^{n-1}F^j(J)\subset [a+1,b]$; 
		\item $J\subset (b-1,a+1)$ , $\bigcup_{j=0}^{n-1}F^j(J)\subset (b-1,a+1)$.	
	\end{enumerate}
\end{lem}
\begin{proof}
	Assume that $J\cap [a,b-1]\neq \emptyset$.
	If $J\cap (b-1,b]\neq \emptyset$, then
	$b-1\in J$. Therefore for any $k\in\bbn$, $f^k(J)\cap [a,b-1]\neq\emptyset$.
	In particular, $f^{n-1}(J)\cap [a,b-1]\neq\emptyset$.
	Since $J \cap f^{n-1}(J)=\emptyset$, $b-1\not\in f^{n-1}(J)$,
	thus $f^{n-1}(J)\subset [a,b-1)$ and $f^n(J)=f(f^{n-1}(J))\subset [a,b-1]\neq J$.
	Which is a contradiction since $J$ is a periodic interval with period $n$. Therefore $J\subset [a,b-1]$.
	
    Similarly, if $J\cap [a+1,b]\neq \emptyset$, then $J\subset [a+1,b]$.
	
	Since $f^i(J)$ is a periodic interval with period $n$, by the above discussion we can get the conclusion.
\end{proof}

\begin{lem}
Let $f\in\css$ with $\htop(f)=0$, $\degg(f)=1$ and $\fix(f)\neq\emptyset$. Let $F$ and $I=[a,b]$ be a lifting and an interval provided by Lemma~\ref{lem:I-extension}. If $K$ is a periodic interval of $\bbs$ with period $n>1$. Then there exists a periodic interval $J$ of $I$ with period $n$ such that $e(J) = K$.
\end{lem}
\begin{proof}
Since $K,f(K),\cdots ,f^{n-1}(K)$ are mutually disjoint and $|K|<1$, there exists an interval $J\subset [a,a+1)$ or $J\subset (b-1,b]$ such that 
$e(J) = K$. Without loss of generality, we assume that $J\subset [a,a+1)$.
Since $e\circ F^i=f^i\circ e$,  $J,F(J),\cdots ,F^{n-1}(J)$ are mutually disjoint and $e(F^n(J))=f^n(e(J))=f^n(K)=K=e(J)$.

If $J\cap [a,b-1]\neq\emptyset$ and $J\cap (b-1,a+1)\neq\emptyset$,
then $b-1\in J$. Since $F([a,b-1])\subset [a,b-1]$, 
for any $i$, $F^i(J)\cap [a,b-1]\neq\emptyset$ and in particular $F^{n-1}(J)\cap [a,b-1]\neq\emptyset$.
Note that $J\cap F^{n-1}(J)=\emptyset$, $b-1\not\in F^{n-1}(J)$,
and $F^{n-1}(J)$ is an interval, then $F^{n-1}(J)\subset [a,b-1)$.
But $F([a,b-1])\subset [a,b-1]$ and $F^{n}(J)\subset[a,b-1]$ which is a contradiction since $e(F^n(J))=K=e(J)$. Therefore either $J\subset(b-1,a+1)$ or $J\subset [a,b-1]$.

Since both $e|_{(b-1,a+1)}$ and $e|_{[a,b-1]}$ are one to one,
$F^n(J)=e^{-1}|_{(b-1,a+1)}(K)=J$ and $F^n(J)=e^{-1}|_{[a,b-1]}(K)=J$. In conclusion $J$ is a periodic interval with period $n$.

\end{proof}

\section{The equivalence of non-separable pairs, IT-pairs and IN-pairs}\label{section:Circle map and its Lift}
In this section we first introduce non-separable pairs for both interval maps and circle maps. Then we study the relationship between the non-separable pairs of $f$ and $f^p$ and the relationship between the non-separable pairs of $f$ and its lifting.

\begin{defn}\label{defn: non-sep-1}
	Let $f\in C(M,M)$.
	A pair $\langle x,y\rangle \in M\times M$ is called \emph{separable}
	if there exist two periodic intervals $J_1$ and $J_2$ such that $J_1\cap J_2=\emptyset$, $x\in J_1$ and $y\in J_2$.  
	A pair $\langle x,y\rangle \in M\times M$ with $x\neq y$
	is called \emph{non-separable} if there exists $z\in M$ such that $x,y\in \omega(z,f)$ and $\langle x,y\rangle$ is not separable.
	Denote by $NS(M,f)$ the set of all non-separable pairs of $f$.
\end{defn}

Note that we require that the points in a non-separable pair are in 
the same $\omega$-limit set.
We have the following characterization of non-separable pairs.

\begin{lem}\label{lem:non-sep}
	Let $f\in C(M,M)$, $\htop(f)=0$ and $\per(f)\neq\emptyset$.
	Assume that $x\neq y\in M$ and there exists $z\in M$ such that 
	$x,y\in \omega(z,f)$.
	Then $\langle x, y\rangle $ is non-separable if and only if there exists a sequence of periodic intervals $\{K_n\}_{n=1}^{\infty}$  with periods $\{m_n\}_{n=1}^\infty$ such that $K_{n+1}\subset K_{n}$, 
	$\lim_{n\to\infty}m_n=\infty$ and $x,y\in K_n$ for all $n\in\bbn$.
\end{lem}
\begin{proof} 
	If $\omega(z,f)$ is finite, then it is a periodic orbit.
	Therefore $\langle x,y\rangle$ is separable and the result is clear.
	Now we consider the situation when $\omega(z,f)$ is infinite.
	Let $(J_n)_{n\in\bbn}$ and $(k_n)_{n\in\bbn}$ 
	provided by Lemma~\ref{rem: period portion} for the period portion of $\omega(z,f)$.
	
	$(\Rightarrow)$ 
	If $\langle x,y\rangle $ is non-separable,
	then for each $n\in\bbn$ there exists $i_n\in\{0,1,\dotsc,k_n-1\}$
	such that $x,y\in f^{i_n}(J_n)$.
	Thus $(J_n)_{n\in\bbn}$ and $(k_n)_{n\in\bbn}$  are as required.
	
	$(\Leftarrow )$ 
	Assume that $\langle x,y\rangle$ is separable, that is, 
	there exist two periodic intervals $M_1$ and $M_2$ such that $M_1\cap M_2=\emptyset$, $x\in M_1$ and $y\in M_2$. 
	Denote by $p_1$ and $p_2$ the period of $K_1$ and $K_2$ respectively.
	By Lemma~\ref{lem:periodic-splitting-intersect}, 
	there exist  $n_1,n_2\in\bbn$ such that 
	$\bigcup_{i=0}^{k_{n_1}-1}f^i(J_{n_1})\subset \bigcup_{i=0}^{p_1-1}f^i(M_1)$
	and $\bigcup_{i=0}^{k_{n_2}-1}f^i(J_{n_2})\subset \bigcup_{i=0}^{p_2-1}f^i(M_2)$.
	Without loss of generality, we assume that $n_1\leq n_2$, then $J_{n_1}\supset J_{n_2}$.
	There exist unique $i_1, i_2\in \{0,1,\dotsc,k_{n_2}-1\}$ such that
	$x\in f^{i_1}(J_{n_2})$ and $y\in f^{i_2}(J_{n_2})$.
	As $x\in M_1$, $y\in M_2$ and $M_1\cap M_2=\emptyset$, we have $i_1\neq i_2$.
 By Lemma \ref{rem: period portion}, $x$ is not eventually periodic.
 For each $n\in\bbn$, 
 there exists $j\in\bbn$ such that $f^j(x)$ are in the interior of 
 $\bigcup_{i=0}^{m_n-1}f^i(K_n)$ and then $\omega(z,f)\subset  \bigcup_{i=0}^{m_n-1}f^i(K_n)$. 
 By Lemma 2.8(5) of \cite{LOZ18}, 
 there exists $n_3\in\bbn$ such that $\bigcup_{i=0}^{m_{n_3}-1}f^i(K_{n_3})\subset \bigcup_{i=0}^{k_{n_2}-1}f^i(J_{n_2})$. 
 Note that $x,y\in K_{n_3}$ and then there exists $i_0\in \{0,1,\dotsc,k_{n_2}-1\}$
 such that $x,y\in f^{i_0}(J_{n_2})$. This is a contradiction.
\end{proof}

\begin{lem}\label{lem:non-sep-p}
	Let $f\in C(M,M)$, $\htop(f)=0$ and $\per(f)\neq\emptyset$.
	Then for $p\in\bbn$ which is the minimal period provided by Theorem \ref{thm:deg-1-period}, $NS(M,f)=NS(M,f^p)$.
\end{lem}

\begin{proof}
	$(\Rightarrow)$ 
	Let $\langle x,y\rangle \in NS(M,f)$.
	By the definition there exists $z\in M$ such that $x,y\in \omega(z,f)$.
	Let $(J_n)_{n\in\bbn}$ and $(k_n)_{n\in\bbn}$ 
	provided by Lemma~\ref{rem: period portion}.
For every $n\in\bbn$, there exists $i_n\in\{0,1,\dotsc,k_n-1\}$ such that 
	$x,y\in f^{i_n}(J_n)$.
It is easy to see that for each $j\in \left\lbrace 0,1,\dotsc,k_n-1\right\rbrace $, $f^{j}(J_n)$ is a periodic interval 
of $f^p$ with period $\lcm(k_n,p)/p$.
In particular, $f^{i_n}(J_n)$ is a periodic interval of $f^p$.
Now we show that there exists $j_0\in\bbn$ such that $x,y\in \omega(f^{j_0}(z),f^p)$. 

By Lemma~\ref{thm:interval-entropy-0}, Lemma~\ref{thm:deg-1-period}
and Remark~\ref{rem:deg-1-period-interval},
there exists some $n$ large enough such that $p$ divides $k_n$.
Without loss of generality, $p$ dividess $k_1$. Denote $m=\frac{k_1}{p}$.
If $x\in \omega(f^{i_0}(z),f^p)$,  $y\in \omega(f^{j_0}(z),f^p)$ where $i_0\neq j_0$,
by Lemma~\ref{rem: period portion}, we can assume that $\omega(z,f^{k_1})\subset J_1$.
Then
\[\omega(f^{i_0}(z),f^p) = \bigcup_{j=0}^{m-1} f^{i_0+jp}(\omega(z,f^{k_1})) \subset \bigcup_{j=0}^{m-1} f^{i_0+jp}(J_1) \]
and 
\[\omega(f^{j_0}(z),f^p) = \bigcup_{j=0}^{m-1} f^{j_0+jp}(\omega(z,f^{k_1}))\subset \bigcup_{j=0}^{m-1} f^{j_0+jp}(J_1).\]
Therefore $\omega(f^{i_0}(z),f^p)\cap \omega(f^{j_0}(z),f^p)=\emptyset$,
which is a contradiction since $x,y\in f^{i_1}(J_1)$.
Thus there exists $j_0\in\bbn$ such that $x,y\in \omega(f^{j_0}(z),f^p)$.

We can get the conclusion that $\left\langle x,y \right\rangle$ is a non-separable pair of $(M, f^p)$.

$(\Leftarrow)$ Let $\langle x,y\rangle \in NS(M,f^p)$.
By the definition there exists some $z\in\bbs$ such that $x,y\in \omega(z,f^p)$, therefore $x,y\in \omega(z,f)$.
If $\langle x, y\rangle$ is separable for $f$,
then there exist two periodic intervals $J_1$ and $J_2$ for $f$ such that $J_1\cap J_2=\emptyset$, $x\in J_1$ and $y\in J_2$. 
Note that $J_1$ and $J_2$ are also periodic intervals for $f^p$ and then  $\langle x,y\rangle$ is separable for $f^p$.
By Lemma~\ref{lem:non-sep}, this is a contradiction.
Thus $\langle x, y\rangle$ is not separable for $f$ and by lemma~\ref{lem:non-sep} again $\langle x,y\rangle \in NS(M,f)$.
\end{proof}

%Another approach

%separable pair in Kuchta,

%and show that under the condition two points is in the same omega limit set, then a pair is not separable if and only if it is non-separable.

The following lemma characterizes the relationship between the non-separable pairs of circle map and its lifting. 

\begin{lem}\label{lem:non-sep-e}
Let $f\in \css$, $\htop(f)=0$ and $\fix(f)\neq\emptyset$.
Denote $F$ and $I$ provided by the Lemma~\ref{lem:I-extension}.
Then $e\times e(NS(I,F))=NS(\bbs,f)$.
\end{lem}
\begin{proof}
By Lemma~\ref{lem:zero-entropy-non-ext}, 
$|\degg(f)|\leq 1$ and $f$ is non-extensible.
We divide the proof into three cases according to the degree of $f$. 

\textbf{Case 1}: $\degg(f)=0$.
As $|I|=1$, 
it is easy to see that if $J$ is a periodic interval of $I$
with period $n\geq 2$, then $e(J)$ is also a periodic interval of $\bbs$ with period $n$.
And if $K$ is a periodic interval of $\bbs$
with period $n\geq 2$, then we can pick a periodic interval $\widetilde{K}$ of $I$ such that $e(\widetilde{K})=K$.
As $e$ is a homeomorphism when it is restricted to an interval with length less than $1$,  by the definition of non-separable pairs, it is not hard to see that
 $e\times e(NS(I,F))=NS(\bbs,f)$.
 
\textbf{Case 2}: $\degg(f)=1$.
Assume that $\langle x,y\rangle \in NS(I,F)$.
Since $x,y\in \omega(z,F)$ for some $z\in I$, by Lemma~\ref{lem:W-limit set-conjugate}, $e(x),e(y)\in \omega(e(z),f)$. 
By the definition,  there exists a sequence of periodic intervals of $(I,F)$, $\{J_n\}_{n=1}^{\infty}$  with periods $\{m_n\}_{n=1}^\infty$ such that $J_{n+1}\subset J_{n}$, 
$\lim_{n\to\infty}m_n=\infty$ and $x,y\in J_n$ for all $n\in\bbn$. 
Since $m_n\to\infty$, we can assume that $m_n>1$, by Lemma~\ref{lem:F-periodic-subinterval}, 
for any $n\in\bbn$, $\bigcup_{j=0}^{m_n-1} F^j(J_n)\subset [a,b-1]$,
$\bigcup_{j=0}^{m_n-1} F^j(J_n)\subset [a+1,b]$, or 
$\bigcup_{j=0}^{m_n-1} F^j(J_n)\subset (b-1,a+1)$. 
Since $|[a,b-1]|<1$, $|[a+1,b]|<1$ and $|(b-1,a+1)|\leq 1$. Denote $J_n'=e(J_n)$,
it is obvious that $\{f^j(J_n')\}_{j=0}^{m_n-1}$ form a cycle of intervals of $\mathbb{S}$.
Therefore there exists a sequence of periodic intervals of $(\bbs,f)$, $\{J_n^{'}\}_{n=1}^{\infty}$  with periods $\{m_n\}_{n=1}^\infty$ such that $J_{n+1}'\subset J_{n}'$, 
$\lim_{n\to\infty}m_n=\infty$ and $e(x), e(y)\in J_n'$ for all $n\in\bbn$.
We conclude that $\left\langle e(x), e(y) \right\rangle $ is a non-separable pair of $(\bbs,f)$. 

Now we assume that $\langle x, y\rangle\in NS(\bbs,f)$.
By the definition, there exists $z\in\bbs$ such that $x,y\in \omega(z,f)$ and there exists a sequence of periodic intervals $\{J_n\}_{n=1}^{\infty}$ of $(\bbs,f)$ with period $\{m_n\}_{n=1}^\infty$ such that $J_{n+1}\subset J_{n}$,  
$\lim_{n\to\infty}m_n=\infty$ and for any $n\in\bbn$, $x,y\in J_n$.
Since $m_n\to \infty$, we can assume that $m_n>1$. 
According to Lemma~\ref{lem:F-periodic-subinterval}, for any $n$, 
there exists a periodic interval $J_n'$ with period $m_n$ of $I$ such that $e(J_n')=J_n$.
By lemma~\ref{lem:F-periodic-subinterval}, $J_n'\subset [a,b-1]$, $J_n'\subset (b-1,a+1)$ or $J_n'\subset [a+1,b]$ for all $n$.
By taking subsequence, we can assume that for any $n$, $J_n'\subset [a,b-1]$.
There exists some $k$ large enough such that $f^k(z)\in J_1$.
Thus we can take $z'\in J_1'$ such that $e(z')=f^k(z)$ and as a result $e(\omega(z',F))=\omega(f^k(z),f)=\omega(z,f)$.
Taking $x',y'\in  J_1'$ such that $e(x')=x$ and $e(y')=y$.  
Since $|b-1-a|<1$ and $e$ is a homeomorphism when it is restricted to an interval whose length is less than $1$,
$x',y' \in \omega(z',F)$ and $x',y'\in J_n'$.
We get the conclusion that $\left\langle x', y' \right\rangle $ is a non-separable pair of $(I,F)$.

\textbf{Case 3}: $\degg(f)=-1$. 
By Lemma~\ref{lem:non-sep-p}, $NS(\bbs,f)=NS(\bbs,f^2)$
and $NS(I,F^2)=NS(I,F)$.
Note that $\degg(f^2)=1$ and $F^2$ is a lifting of $f^2$. By Case 2, 
$e\times e(NS(I,F^2))=NS(\bbs,f^2)$. 
Thus, $e\times e(NS(I,F))=NS(\bbs,f)$.
\end{proof}

We will need the following result for interval maps.
\begin{thm}[\cite{Li J 2011}]\label{thm:interval-NS-IN-IT}
	Let $f\in C(I,I)$ and $\htop(f)=0$.
	Then $NS(I,f)\setminus\Delta=IT(I,f)\setminus\Delta=IN(I,f)\setminus\Delta$.
\end{thm}

Now we are ready to prove our first main result (Theorem \ref{thm:main-result1}).
Using notations introduced in sections 2 and 3, we can restate
it as follows. 

\begin{thm}\label{thm:IN=IT=NS}
Let $f\in \css$ and $\htop(f)=0$.
Then  $NS(\bbs,f)\setminus\Delta=IT(\bbs,f)\setminus\Delta=IN(\bbs,f)\setminus\Delta$.
\end{thm}
\begin{proof}
 
Since $\htop(f)=0$,  if $f$ has fixed point, by Lemma~\ref{lem:I-extension}, we can find a lifting $F$ of $f$, such that there exists an interval $I$ which satisfys that $1\leq\left| I\right|<2 $ and 
$F|_I$ is an interval map. By Lemma~\ref{lift-entropy}, $\htop(F|_I)\leq \htop(f)=0$, that is $\htop(F|_I)=0$. 

\textbf{$NS(\bbs,f)\subset IN(\bbs,f)$:}
Assume that $\left\langle x,y \right\rangle \in NS(\bbs,f)$,
by Lemma~\ref{lem:non-sep-e},  there exists a non-separable pair
$\left\langle x_1,y_1 \right\rangle $ of $(I,F)$ such that
$\pi(x_1)=x$ and $\pi(y_1)=y$.
By Lemma~\ref{thm:interval-NS-IN-IT}, $\left\langle x_1,y_1 \right\rangle $ is an IN-pair of $(I,F)$. And by Lemma~\ref{INIT-conjugate}, $\left\langle x,y \right\rangle$ is an IN-pair of $(\bbs,f)$.

\textbf{$IN(\bbs,f)\subset IT(\bbs,f)$:}
Assume that $\left\langle x,y \right\rangle \in NS(\bbs,f)$,
by Lemma~\ref{INIT-conjugate},  
there exists a IN-pair
$\left\langle x_1,y_1 \right\rangle $ of $(I,F)$ such that
$\pi(x_1)=x$ and $\pi(y_1)=y$. By Lemma~\ref{thm:interval-NS-IN-IT}, 
$\left\langle x_1,y_1 \right\rangle $ is an IT-pair of $(I,F)$. And by Lemma  ~\ref{INIT-conjugate},  $\left\langle x,y \right\rangle$ is an IT-pair of $(\bbs,f)$.

\textbf{$IT(\bbs,f)\subset NS(\bbs,f)$:}
Assume that $\left\langle x,y \right\rangle \in NS(\bbs,f)$,
by  Lemma~\ref{INIT-conjugate}, 
there exists a IT-pair
$\left\langle x_1,y_1 \right\rangle $ of $(I,F)$ such that
$\pi(x_1)=x$ and $\pi(y_1)=y$.
By lemma~\ref{thm:interval-NS-IN-IT},  
$\left\langle x_1,y_1 \right\rangle $ is an non-separable pair of $(I,F)$. 
And by Lemma~\ref{lem:non-sep-e}
$\left\langle x,y \right\rangle $ is an non-separable pair of $(\bbs,f)$.

If $\per(f)\neq \emptyset$, without loss of generallity we assume that $f$ has a periodic point whose period is $p$, by the proof of the situation when $\fix(f)\neq \emptyset$,  $NS(\bbs,f^p)\setminus\Delta=IT(\bbs,f^p)\setminus\Delta=IN(\bbs,f^p)\setminus\Delta$.
By Lemma~\ref{lem:non-sep-p} and Lemma~\ref{IN-IT-p},   $NS(\bbs,f)\setminus\Delta=IT(\bbs,f\setminus\Delta)=IN(\bbs,f)\setminus\Delta$.

If $\per(f)=\emptyset$, by Lemma~\ref{lem:f-entropy-no-period-point}, $h_{top}^{\infty}(f)=0$,
therefore $f$ has no IN-pairs. By the definition of IN-pair and IT-pair, $f$ has no IT-pairs.
Besides, no periodic points implies no cycle of intervals which in turn implies that there is no non-separable pairs.
\end{proof}

%\begin{cor}
%Let $f\in C(M,M)$ and $\htop(f)=0$.
%Then $NS(M,f)=NS(M,f^n)$ for all $n\in\bbn$.	
%\end{cor}

%\begin{cor}\label{lem:f-null-F}
%Let $f\in \css$, $\htop(f)=0$ and $\fix(f)\neq\emptyset$.
%Let $F$ and $I$ as in the Lemma~\ref{lem:I-extension}.
%Then $(\bbs,f)$ is null if and only if $(I,F)$ is null.	
%\end{cor}

\section{Maximal pattern entropy and topological null systems}
\label{section:Maximal pattern} 
In this section we first introduce maximal pattern entropy and some related concepts. 
%Lemma~\ref{lem:F-poly-f} investigate the relationship between the maximal pattern entropy of $f$ and $f^k$ and the relationship between the maximal pattern entropy of the two topological conjugate systems.
%Lemma~\ref{lem:open-cover, spanning-set and separated-set} study the least cardinality of the subcover, the least cardinality of the spanning sets and the maximal cardinality of the separated sets.
%Proposition~\ref{prop:poly-order-2} shows that the maximal pattern entropy is of polynomial order if and only if the least cardinality of the spanning sets is of polynomial order if and only if the maximal cardinality of the separated sets is of polynomial order.
%Lemma~\ref{lem:f-null-F} tells us that $f$ is null implies its lifting is null in a particular situation.

\begin{defn} 
	Let $(X,f)$ be a topological dynamical system and
	$\calu\in \calc_X^o$. 
	The \emph{maximal pattern entropy of $(X,f)$} with respect to  $\calu$ is defined as  
	\[\htop ^*(f, \calu)=\lim_{n \to \infty}\frac{1}{n}\log(p^*_{X,\calu}(n))\]
where
	\[p^*_{X,\calu}(n)=\max_{( t_{1},t_{2},t_{3},\cdots,t_{n}) \in \bbz_+^n}  N \bigg (\bigvee_{i=1}^{n}f^{-t_i}\calu \bigg ),\]
and the \emph{maximal pattern entropy} of $(X,f)$ is  	
	\[\htop^*(f)=\sup_{\calu\in\calc_X^o}\htop^*(f, \calu).\]
In \cite{Huang W Ye X}, Huang and Ye showed that 
$\htop^*(f)=\htop^\infty(f)$. 
\end{defn}

\begin{defn}
	Let $(X,f)$ be a topological dynamical system and $\calu\in\calc_X^0$.
	We say that $p^*_{X,\calu}$ is of polynomial order
	if there exists $C>0$ such that $p^*_{X,\calu}(n)\leq Cn^C$ for all $n\in\bbn$.
\end{defn}

We first have the following observation.

\begin{lem}\label{lem:F-poly-f}
	\begin{enumerate}
		\item Let $(X,f)$ be a topological dynamical system and $k\in\bbn$.
		Then  $p^*_{X,\calu}$ with respect to $f$ is of polynomial order 
		for all $\calu\in\calc_X^o$ if and only if so is with respect to $f^k$.
		\item Let $e\colon (X,f)\to (Y,g)$ be a factor map of two dynamical systems.
		If $p^*_{X,\calu}$  is of polynomial order 
		for all $\calu\in\calc_X^o$
		then  $p^*_{Y,\calv}$  is of polynomial order 
		for all $\calv\in\calc_Y^o$.
	\end{enumerate}	
\end{lem}
\begin{proof}
	(1) Let $\calu\in\calc_X^o$. $p^*_{X,\calu}$ with respect to $f$ and $f^k$ are denoted by $p^*_{X,\calu,f}$
	and $p^*_{X,\calu,f^k}$ respectively.
	It is clear that $p^*_{X,\calu,f^k}(n)\leq p^*_{X,\calu,f}(kn)$ for all $n\in\bbn$.
	If $p^*_{X,\calu,f}$ is of polynomial order, then
	so is $p^*_{X,\calu,f^k}$.
	On the other hand, 
	$p^*_{X,\calu,f}(n)\leq  p^*_{X,\bigvee_{i=0}^{k-1}f^{-i}(\calu),f^k}(n)$ for all $n\in\bbn$.
	If $p^*_{X,\bigvee_{i=0}^{k-1}f^{-i}(\calu),f^k}$ is of polynomial order, then
	so is $p^*_{X,\calu,f}$.
	
	(2) Let $\calv\in \calc_Y^o$.
	Then $\pi^{-1}(\calv)\in\calc_X^o$.
	It is clear that $p^*_{Y,\calv,g}(n)\leq p^*_{X,\pi^{-1}(\calv),f}(n)$ for all $n\in\bbn$.
	If $p^*_{X,\pi^{-1}(\calv),f}$ is of polynomial order, then
	so is $p^*_{Y,\calv,g}$.
\end{proof}

\begin{rem}
	It is clear that if $p^*_{X,\calu}$ is of polynomial order
	then $\htop ^*(f, \calu)=0$.
	It is shown in \cite{Huang W Ye X} that 
	if $\calu$ is a clopen (i.e. closed and open) partition 
	of $X$ then $\htop ^*(f, \calu)=0$
	implies that $p^*_{X,\calu}(n)$ is of polynomial order.
	As a consequence, if $X$ is zero dimensional, 
	then $(X,f)$ is null if and only if 
	$p^*_{X,\calu}(n)$ is of polynomial order
	for all $\calu\in\calc_X^o$.
\end{rem}

\begin{defn}\label{defn:spanning-set}
	Let $(X,f)$ be a topological dynamical system
	and $A=\lbrace a_1,a_2,\dotsc  \rbrace$
	be an increasing sequence of non-negative integers.
	For $\epsilon>0$, 
	a subset $E$ of $X$ is called a \emph{$(A,n,\epsilon)$-spanning set} if for any $x\in X$, 
	there is some $y\in E$ such that 
	$d(f^{a_{i}}(x),f^{a_{i}}(y)) <\epsilon$
	for all $i\in {1,2,\dotsc,n}$.
	We define $r_A(n,\epsilon)$ as the minimal cardinality
	of $(A,n,\epsilon)$-spanning sets of $X$ 
	and 
	\[r^*(n,\epsilon)=\sup_A r_A(n,\epsilon),\]
	where the supreme ranges over all  increasing sequences of non-negative integers.
	
	We say that a subset $F$ is \emph{$(A,n,\epsilon)$-separated set} of $X$
	if for any $x,y\in F$, there is some  $i\in\{1,2,\dotsc,n\} $ such that $d(f^{a_{i}}(x),f^{a_{i}}(y))\geq\epsilon$.
	We define $s_A(n,\epsilon)$ as the maximal cardinality
	of $(A,n,\epsilon)$-separated subsets of $X$ 
	and 
	\[s^*(n,\epsilon)=\sup_A s_A(n,\epsilon),\]
	where the supreme ranges over all  increasing sequences of non-negative integers.
	
	It is clear that $r_A(n,\epsilon)\leq s_A(n,\epsilon)$
	and then $r^*(n,\epsilon)\leq s^*(n,\epsilon)$.
\end{defn}

%\textcolor{red}{Exact reference of the following result!} 
%Ergodic Theory Proposition 5.2.2 is OK???
The following lemma is standard in topological dynamics, see e.g. \cite[Theorem 7.7]{W82}.

\begin{lem}\label{lem:open-cover, spanning-set and separated-set}
Let $(X,f)$ be a topological dynamical system and $A=\left\lbrace a_1,a_2,\dotsc\right \rbrace$
be an increasing sequence of non-negative integers.
\begin{enumerate}
	\item Let $\calu\in \calc_X^o$ and $\delta$ be a Lebesgue number of  $\calu$.
	Then for every $n\in\bbn$, 
	\[N\biggl(\bigvee_{i=1}^n f^{-t_i} (\calu)\biggr)\leq r_A(n,\frac{\delta}{2})\leq s_A(n,\frac{\delta}{2}).\]
	\item Let $\epsilon>0$ and choose $\calv\in \calc_X^o$ with a Lebesgue number less than $\epsilon$.
	Then 
	\[
	r_A(n,\epsilon)\leq s_A(n,\epsilon)\leq N\biggl(\bigvee_{i=1}^n f^{-t_i} (\calv)\biggr).
	\]
\end{enumerate}
\end{lem}

By the definition of maximal pattern entropy and the above lemma, we have the following remark.
\begin{rem}\label{prop:poly-order-2}
Let $(X,f)$ be a topological dynamical system.
Then 
\[
\htop^*(X,f)=\lim_{\epsilon\rightarrow 0} \limsup_{n\to\infty}\frac{1}{n} \log r^*(n,\epsilon)=\lim_{\epsilon\rightarrow 0} \limsup_{n\to\infty}\frac{1}{n} \log s^*(n,\epsilon),
\]
and the following statements are equivalent:
\begin{enumerate}
	\item $p^*_{X,\calu}$ is of polynomial order for all $\calu\in\calc_X^o$;
	\item $r^*(\,\cdot\,,\epsilon)$ is of polynomial order for all $\epsilon>0$;
	\item $s^*(\,\cdot\,,\epsilon)$ is of polynomial order for all $\epsilon>0$.
\end{enumerate}
\end{rem}

\begin{lem}\label{lem:f-null-F}
Let $f\in \css$, $\htop(f)=0$ and $\fix(f)\neq\emptyset$.
Denote $F$ and $I$ provided by Lemma~\ref{lem:I-extension}.
Then $(\bbs,f)$ is null if and only if $(I,F)$ is null.	
\end{lem}
\begin{proof}
	$(\Leftarrow)$
	If $(I,F)$ is null, by lemma~\ref{lem: conjugate and entropy},
	$(\bbs,f)$ is null.
	
	$(\Rightarrow)$
	If $(\bbs,f)$ is null, 
	by lemma~\ref{thm:null-IN},
	$f$ has no non-diagonal IN-pairs.
	By lemma~\ref{thm:IN=IT=NS}, $f$ has no non-separable pairs.
	By lemma~\ref{lem:non-sep-e}, $F$ has no non-separable pairs, thus $F$ is not Li-Yorke chaotic and by lemma~\ref{lem:f-entropy-no-period-point} $F$ is null.
\end{proof}

We will also need the following result.

\begin{thm}[\cite{Li J 2011}] \label{thm:interval-null-poly}
Let $f\in C(I,I)$. 
Then
	$f$ is null if and only if for any $\calu\in\calc_X^o$, 
	$p^*_{X,\calu}(n)$ is of polynomial order.
\end{thm}

Now we are ready to prove our second main result.

\begin{proof}[Proof of Theorem~\ref{thm:main-result2}]
The sufficiency of the proof is clear, we only proof the necessity.
First we assume that $\per(f)\neq\emptyset$.
Then there exists $p\in\bbn$ such that $\fix(f^p)\neq\emptyset$.
Let $F$ and $I$ provided by Lemma~\ref{lem:I-extension} for $f^p$.
By Lemma~\ref{lem:f-null-F}, $(I,F)$ is null.
Then by Lemma~\ref{thm:interval-null-poly}, $p^*_{I,\calu}(n)$ is of polynomial order for all $\calu\in\calc_I^o$.
As $e\colon (I,F)\to (\bbs,f^p)$ is a factor map,
by Lemma~\ref{lem:F-poly-f}(2), $p^*_{\bbs,\calu}$ with respect to $f^p$ is of polynomial order for all $\calu\in\calc_\bbs^o$.
Thus by Lemma~\ref{lem:F-poly-f}(1),  $p^*_{\bbs,\calu}$ with respect to $f$ is of polynomial order for all $\calu\in\calc_\bbs^o$.

Now we consider the situation when $\per(f)=\emptyset$. We divide the proof into two cases.

\textbf{Case 1}: $f$ is homeomorphism.

Fix $\epsilon>0$.
Let $A=\left\lbrace a_1,  a_2, \cdots \right\rbrace $ be a  sequence of non-negative integers
and $F=\left\lbrace x_1,  x_2, \cdots , x_k \right\rbrace $ be a
$(A,n, \epsilon)$-separated set of $(\bbs,f)$. 
Without loss of generality, we assume that 
$x_1<x_2<\cdots <x_k<x_1 $ under the ordering of $\bbs$ in Definition~\ref{defn:order of circle}.
Obviously in this case $f$ is an orientation preserving and  $ f^{a_i}(x_1)< f^{a_i}(x_2)< \cdots < f^{a_i}(x_k)<f^{a_i}(x_1)$ under the ordering of $\bbs$ for any $ i\in\{1,2,\cdots, n\}$.
Therefore for each $i$,
\[\sum_{j=1}^{k-1} d(f^{a_i}(x_j), f^{a_i}(x_{j+1}))+ d(f^{a_i}(x_k), f^{a_i}(x_1))\leq 1\] under the metric $d$ which is defined  in Definition~\ref{defn:metric of circle}.  
Since $F=\left\lbrace x_1,  x_2, \cdots , x_k \right\rbrace $ is $(A,n, \epsilon)$-separated, for any pair 
$\left\lbrace x_j, x_{j+1} \right\rbrace $ or
$\left\lbrace x_k, x_1 \right\rbrace $, where $j\in \left\lbrace 1,2,\cdots, k-1 \right\rbrace$,
there exists some 
$i\in \left\lbrace 1,2,\cdots, n \right\rbrace $
such that 
$d(f^{a_i}(x_j),f^{a_i}(x_{j+1}))>\epsilon$, where $j\in \left\lbrace 1,2,\cdots, k \right\rbrace$ and $x_{k+1}=x_1$.
Therefore the number of  the pairs $\left\lbrace x_j, x_{j+1} \right\rbrace $
cannot exceed $n(\left[ \frac{1}{\epsilon}\right]+1) $ since if the number of the pairs exceeds $n(\left[ \frac{1}{\epsilon}\right]+1) $, then there exists some $i$
such that the number of pairs $\left\lbrace f^{a_i}(x_j),f^{a_i}(x_{j+1})\right\rbrace $ which satisfy that 
$d(f^{a_i}(x_j),f^{a_i}(x_{j+1}))>\epsilon$ exceeds $\left[ \frac{1}{\epsilon}\right]+1$,
thus \[\sum_{j=1}^{k-1} d(f^{a_i}(x_j), f^{a_i}(x_{j+1}))+ d(f^{a_i}(x_k), f^{a_i}(x_1))> 1,\] which is a contradiction.
Therefore, $k\leq n(\left[ \frac{1}{\epsilon}\right]+1)$ 
and then $s_A(n,\epsilon)\leq n(\left[ \frac{1}{\epsilon}\right]+1)$ since $F$ is arbitrary.
Moreover, $s^*(n,\epsilon)\leq n(\left[ \frac{1}{\epsilon}\right]+1)$ since the upper bound $n(\left[ \frac{1}{\epsilon}\right]+1)$ does not depend on the sequence $A$, that is $s^*(\,\cdot\,,\epsilon)$ is of polynomial order.
As $\epsilon$ is arbitrary, by Proposition~\ref{prop:poly-order-2},
$p^*_{\bbs,\calu}$ 
is of polynomial order for all $\calu\in \calc_\bbs^o$.

\textbf{Case 2}: $f$ is not homeomorphism.

By Lemma~\ref{lem:not-homermorphism and nowhere dense limit set},  
$W=\omega(x,f)$ is the only $\omega$-limit set of $\bbs$ and it is a nowhere dense perfect set. Moreover, all the contiguous intervals are wandering and the image of any contiguous interval is either a contiguous interval or a point from $W$. By the properties of $\omega$-limit set, $f(W)=W$.	
	
Fix $\epsilon>0$. Let $A=\left\lbrace a_1,  a_2, \cdots \right\rbrace $ be a sequence of non-negative integers, $I_1, I_2, \cdots, I_k$ be all contiguous intervals longer than $\frac{\epsilon}{2}$ and $E_{W}$ be a $(A,f|_W,n,\frac{\epsilon}{2})$-spanning set. 

First we consider the points whose $(A,f,n)$-trajectory is disjoint from $\bigcup_{i=1}^{k}I_i$ and we take $x$ as such a point. If $x\in W$ then 
$x$ is $(A,f,n,\epsilon)$-spanned by $E_W$. If $x\not\in W$ we put $y$ to be an endpoint of 
the contiguous interval which contains $x$. Then $d(f^{a_i}(x),f^{a_i}(y))\leq \frac{\epsilon}{2}$ for all $1\leq i\leq n$ since the $(A,f,n)$-trajectory of $x$ is in $\mathbb{S}\setminus \bigcup_{i=1}^k I_i$. Since $y\in W$, there is a point $z\in E_W$ such that $y$ is $(A,f,n,\epsilon)$-spanned by $z$, as a conclusion the set $E_W$ $(A,f,n,\epsilon)$-spans all such points $x$.

Now we consider the points whose trajectory meet $\bigcup_{i=1}^{k}I_i$. Fix $N\in\bbn$ such that $N>[\frac{1}{\epsilon}]$.
For $1\leq i\leq n$ and $1\leq j\leq k$,
denote $I(i,j)=\left\lbrace x\colon f^{a_i}(x)\in I_j \right\rbrace $. 
It is obvious that $I(i,j)$ is a contiguous interval. Consider its $(A,f,n)$-trajectory
$\left\lbrace f^{a_1}(I(i,j)), f^{a_2}(I(i,j)), \cdots, f^{a_n}(I(i,j))\right\rbrace $.
Each elements in this trajectory is either a contiguous interval or a point from $W$ and all the elements are pairwise disjoint since $I(i,j)$ is wandering.
At most $k$ of them have lengths greater than or equal to $\epsilon$
since $\left\lbrace  I_i\right\rbrace_{i=1}^k $ are all the contiguous intervals longer than $\frac{\epsilon}{2}$. Cut each of 
such elements to $N$ segments shorter than $\epsilon$. All of the other elements of the trajectory will be consider to be segments themselves. To each $x\in I(i,j)$ assign its 
code--the $n$-tuple $(S_1(x),S_2(x),\cdots, S_n(x))$ where $S_i(x)$ is the segment containing $f^{a_i}(x)$. We have at most $N^k$ different codes and all points with the same code can be $(A,f,n,\epsilon)$--spanned by one point. Therefore there is a set with cardinality at most $N^k$ which $(A,f,n,\epsilon)$--spans $I(i,j)$.

By Lemma~\ref{lem:not-homermorphism and nowhere dense limit set}, $f|_W$ is an orientation preserving, moreover by linear extension of $f|_W$ we obtain a homeomorphism $g\in \css$. 
By the proof of \textbf{Case 1} when $f$ is a homeomorphism, we have 
\[r(f|_W,A,n,\epsilon/2)
=r(g|_W,A,n,\epsilon/2)\leq r(g,A,n,\epsilon/2)
\leq s(g,A,n,\epsilon/2) \leq n[\frac{2}{\epsilon}],\]
therefore $r(f,A,n,\epsilon)\leq r(f|_W,A,n,\epsilon/2)+n\cdot k\cdot N^k\leq n[\frac{2}{\epsilon}]+n\cdot k\cdot N^k$.
Moreover, since $N$ and $k$ are independent of the sequence $A$,
\[
r^*(f,n,\epsilon)=\sup_{A}r(f,A,n,\epsilon)\leq n[\frac{1}{\epsilon}] +n k N^k.
\] 
By Proposition~\ref{prop:poly-order-2},
$p^*_{\bbs,\calu}$ is of polynomial order for each open cover $\calu$.
\end{proof}

\section*{Acknowledgements} 
The author would like to thank Prof.\ Jian Li and Prof.\ Piotr Oprocha
for the careful reading and helpful suggestions. 
This work was supported in part by the National Natural Science Foundation of China (11771264) and 
Guangdong Natural Science Foundation (2018B030306024).

\end{document}